\documentclass[12pt]{amsart}
\usepackage{graphics}
\usepackage[toc,page]{appendix}
\usepackage{amscd,ragged2e}
\usepackage{amsmath}
\usepackage{amssymb}
\usepackage{dsfont}
\usepackage{amsmath}
\usepackage{amsthm}
\usepackage{amssymb,hyperref,dsfont,amsrefs}

\newtheorem{theorem}{Theorem}[section]
\newtheorem{definition}[theorem]{Definition}
\newtheorem{proposition}[theorem]{Proposition}
\newtheorem{corollary}[theorem]{Corollary}

\makeatletter
\newcommand\incircbin
{%
  \mathpalette\@incircbin
}
\newcommand\@incircbin[2]
{%
  \mathbin%
  {%
    \ooalign{\hidewidth$#1#2$\hidewidth\crcr$#1\bigcirc$}%
  }%
}

\makeatother

\begin{document}
\title{The Ergodic Theorem for Random Walks on Finite Quantum Groups}
    \author{J.P. McCarthy}
    \address{Department of Mathematics, Munster Technological University, Cork, Ireland.}

    \email{jeremiah.mccarthy@cit.ie}

    \date{\today}

   \keywords{random walks, finite quantum groups, ergodicity}

\subjclass[2000]{46L53 (60J05, 20G42)}

   \begin{abstract}
Necessary and sufficient conditions for a Markov chain to be ergodic are that the chain is irreducible and aperiodic. This result is manifest in the case of random walks on finite groups by a statement about the support of the driving probability: a random walk on a finite group is ergodic if and only if the support is not concentrated on a proper subgroup, nor on a coset of a proper normal subgroup. The study of random walks on finite groups extends naturally to the study of random walks on finite quantum groups, where a state on the algebra of functions plays the role of the driving probability. Necessary and sufficient conditions for ergodicity of a random walk on a finite quantum group are given on the support projection of the driving state.
\end{abstract}

\maketitle

\section*{Introduction}
Let $\sigma_1,\sigma_2,\dots,\sigma_k$ be a sequence of shuffles of a deck of cards. If the deck starts in some known order, the order of the deck after these $k$ shuffles is given by
$$\Sigma_k=\sigma_k\cdots\sigma_2\cdot \sigma_1.$$
Suppose the shuffles are random variables independently and identically distributed as $\sigma_i\sim\nu\in M_p(S_{52})$, where $M_p(S_{52})$ is the set of probability distributions on $S_{52}$, then $\Sigma_k\sim \nu^{\star k}$ where $\nu^{\star k}$ is defined inductively by
$$\nu^{\star (k+1)}(\{\sigma\})=\sum_{\varrho\in S_{52}}\nu(\{\sigma \varrho^{-1}\})\nu^{\star k}(\{\varrho\}).$$

\bigskip

This generalises to arbitrary finite groups. Given independent and identically distributed $s_i\sim \nu\in M_p(G)$,  consider the random variable:
\begin{equation}\xi_k=s_k\cdots s_2\cdot s_1\,\sim \,\nu^{\star k}.\label{rv}\end{equation}
If the convolution powers $(\nu^{\star k})_{k\geq 1}$ converge to the uniform distribution $\pi\in M_p(G)$, the random walk is said to be \emph{ergodic}.

\bigskip

The study of random walks on finite groups probably has its roots in questions of Markov \cite{Markov} and Borel (and coauthors) \cite{Borel}, who asked which card shuffles would mix up a deck of cards. This qualitative question, the inspiration for the current work, is answered by a folklore theorem, which gives conditions on the support of the \emph{driving} probability that are equivalent to ergodicity:

\begin{theorem}(Ergodic Theorem for Random Walks on Finite Groups)
Let $\nu\in M_p(G)$ be a probability on a finite group $G$. The  associated random walk is ergodic if and only if $\nu$ is not concentrated on a proper subgroup nor the coset of a proper normal subgroup $\bullet$
\end{theorem}

\bigskip

 The study of random walks on finite groups, under the programme of quantum probability, extends in a natural way to the study of random walks on  \emph{quantum} groups. Early work on quantum stochastic processes by various authors led to random walks on duals of compact groups (particularly Biane, see \cite{FG} for references), and other examples, but Franz and Gohm \cite{FG} defined with clarity a \emph{random walk on a (finite) quantum group}.

 \bigskip

The basic qualitative question: what are the conditions on the driving probability for a random walk on a finite quantum group to be ergodic; has remained open since at least 1996 when Pal \cite{Pal} showed that the ergodic theorem for random walks on finite groups does not extend to the quantum setting, that there exist random walks on finite quantum groups that are not ergodic, but neither is the driving `probability' concentrated on any proper quantum subgroup, nor does it have the periodicity associated with being concentrated on a coset of a proper normal subgroup.

\bigskip

The irreducibility condition (see Section \ref{irred}) has received a lot of attention through the study of idempotent states on quantum groups, initiated in \cite{idempotent} on compact quantum groups by Franz, Skalski (and coauthors). This programme of study, particularly \cite{idempotent}, has been cited heavily in this work. To fully adapt the study of idempotent states to irreducible random walks was to prove Proposition \ref{quassup} and Theorem \ref{irrcond}, and these are mostly concerned with introducing to the study of idempotent states the concept of support projections (see Section \ref{supports}).

\bigskip

The work leans most heavily on a paper of Fagnola and Pellicer \cite{FP}, which emerged after the intensive study of idempotent states began. The paper of Fagnola and Pellicer itself follows a paper of Evans and H{\o}egh-Krohn \cite{EVANS}. In this 2009 paper, the notions of irreducibility and periodicity of a stochastic matrix are extended to the case of a unital positive map on a finite-dimensional $\mathrm{C}^*$-algebra, and a noncommutative version of the Perron--Frobenius theorem is given. This current work puts the results of Fagnola and Pellicer in the language of finite quantum groups, and in the language of support projections. The paper of Fagnola and Pellicer is cited so heavily in this work that it will be cited once and for all at this point, with further citations of ``Fagnola and Pellicer'' referring always to \cite{FP}.

\bigskip

A number of partial results, stated for Sekine quantum groups; a sufficient condition of Zhang for aperiodicity \cite{Zhang}, and an ergodic theorem of Baraquin for central states \cite{Baraquin}, have been shown to hold more generally. As the detection of whether or not a random walk satisfies the conditions for ergodicity is non-trivial, partial results such as these are most welcome for any study of random walks on quantum groups.

\bigskip

\bigskip

The paper is organised as follows. In Section 1, finite quantum groups are defined. A number of properties of finite-dimensional $\mathrm{C}^*$-algebras, particularly concerning projections, states, support projections, and densities, are included here also. Finally, the definition of a random walk on a finite quantum group is given. Section 2 takes a brief look at the stochastic operator associated to a random walk, and crucially states the relationship between the distribution of the random walk and powers of the stochastic operator. Results of Fagnola and Pellicer concerning the spectrum of a stochastic operator are stated. In Section 3, irreducible random walks are studied, and the example of Pal discussed in more detail. The programme of study of idempotent states, and their associated group-like projections, is introduced. The definition of irreducible by Fagnola and Pellicer, in the language of harmonic projections, is shown to be equivalent to irreducible (in the sense of an irreducible random walk). Quasi-subgroups are introduced, and it is shown that a random walk concentrated on a proper quasi-subgroup is reducible, and it is shown that this is the only barrier to irreducibility. In Section 4, periodic random walks are studied. This section leans heavily on a result of Fagnola and Pellicer, which says that if an irreducible random walk is not ergodic, there exists a partition of unity that illustrates the periodicity of the walk. It is shown that these projections behave like indicator functions on cosets of proper normal subgroups, that the state defining the random walk is concentrated on one of these projections, and that one of the other projections gives a quasi-subgroup. This allows the Ergodic Theorem for Random Walks on Finite Quantum Groups to be written down. Some partial results are included in Section 5; \ref{KACP} for random walks on Kac--Paljutkin and Sekine quantum groups; \ref{Zhang} for so-called Zhang Convergence; \ref{AmaurySec} for random walks on dual groups; \ref{centralstates} for random walks given by central states.

\section{Preliminaries}

\subsection{Finite Quantum Groups\label{FQG}}
A finite quantum group $\mathbb{G}$ is a so-called \emph{virtual object}, spoken about via its algebra of functions $F(\mathbb{G})$. Here is a definition from \cite{FG}:
\begin{definition}
The \emph{algebra of functions on a finite quantum group}, is a finite-dimensional unital $\mathrm{C}^*$-Hopf algebra $A$; that is a finite-dimensional $\mathrm{C}^*$-algebra $A$ with a $*$-homomorphism $\Delta$, a counit $\varepsilon$, and an antipode $S$, such that $(A,\Delta,S,\varepsilon)$ is a Hopf $*$-algebra.
\end{definition}
Finite quantum groups are generalisations of finite groups in the sense that whenever $A$ is a commutative algebra of functions on a finite quantum group, there exists a finite classical group $G$ such that  $A\cong F(G)$, the algebra of complex-valued functions on $G$. Furthermore whenever $A_1\cong A_2$ are isomorphic as commutative algebras of functions on finite quantum groups, the finite classical groups that $A_1$ and $A_2$ define are isomorphic.  See Section 1.1 of \cite{McCarthy} for a more leisurely introduction to finite quantum groups, and \cite{VD2} for properties.

\bigskip

Denote the algebra of functions on a finite quantum group by $A=:F(\mathbb{G})$, with unit  denoted by $\mathds{1}_{\mathbb{G}}$, and refer to $\mathbb{G}$ as a finite quantum group.

\bigskip

A \emph{projection} in a $\mathrm{C}^\ast$-algebra $A$ is an element $p$ such that $p=p^{*}=p^2$. For a finite (classical) group $G$, every function $G\rightarrow \{0,1\}$ is a projection in $F(G)$. Therefore denote by $2^\mathbb{G}\subset F(\mathbb{G})$ the set of projections in the algebra of functions.

 \bigskip

As a finite-dimensional $\mathrm{C}^*$-algebra, the algebra of functions on a finite quantum group $\mathbb{G}$ is a multi-matrix algebra:
$$F(\mathbb{G})\cong \bigoplus_{i=1}^N M_{n_i}(\mathbb{C}).$$
Its left ideals are of the form $F(\mathbb{G})p$ for $p\in 2^\mathbb{G}$. The central projections  are sums of identity matrices:
\begin{equation}\label{centralproj}Z(F(\mathbb{G}))\cap 2^\mathbb{G}=\left\{\sum_{i=1}^N \alpha_{i}I_{n_i}\,:\,\alpha_i=0,1\right\}.\end{equation}
As the counit is a character, there is a one dimensional factor, whose basis element, a central projection $\eta\in 2^\mathbb{G}$, is called the \emph{Haar} element. By writing, for a general $p\in 2^\mathbb{G}$,
$$p=\alpha \eta\oplus r,$$
and considering $p^2=p$, it follows that $\alpha=0\text{ or }1$.

\subsection{States}
A probability on a classical group $\mu:\mathcal{P}(G)\rightarrow [0,1]$, gives rise to an expectation, also denoted $\mu:F(G)\rightarrow \mathbb{C}$,
$$f\mapsto \sum_{t\in G}f(t)\mu(\{t\}).$$
The expectation is a positive linear functional on $F(G)$ such that $\mu(\mathds{1}_G)=1$. Denoting the set of probabilities on $G$ by $M_p(G)$, this motivates:
\begin{definition}
A \emph{state} $\mu$ on the algebra of functions on a finite quantum group $\mathbb{G}$ is a positive linear functional such that $\mu(\mathds{1}_\mathbb{G})=1$. Denote the set of states on $F(\mathbb{G})$ by $M_p(\mathbb{G})$.
\end{definition}
Define the \emph{convolution of states} $\mu,\,\nu\in M_p(\mathbb{G})$ on a quantum group:
\begin{equation}
\mu\star \nu:=(\mu\otimes \nu)\Delta.\label{conv}\end{equation}
The counit  is a state that is an identity for this convolution:
\begin{equation}\varepsilon\star\mu=\mu=\mu\star\varepsilon \,.\,\,\qquad (\mu\in M_p(\mathbb{G}))\label{eps}\end{equation}
Where $\pi$ is the random/uniform probability on a classical group $G$, consider the state on $F(G)$:

$$f\mapsto \sum_{t\in G} f(t)\,\pi(\{t\})=\frac{1}{|G|}\sum_{t\in G}f(t).$$

This state is called the \emph{Haar} state, and it is \emph{invariant} in the sense that for all $\mu\in M_p(G)$,
\begin{equation}\pi\star \mu=\pi=\mu\star \pi.\label{invar2}\end{equation}
Still in the classical case, this invariance is equivalent to
\begin{equation} \mathds{1}_G\cdot \pi(f)=\left(I_{F(G)}\otimes \pi\right)\Delta(f)= \left(\pi\otimes I_{F(G)}\right)\Delta(f)\,.\,\,\qquad (f\in F(G))\label{invar}\end{equation}

A finite quantum group also has a unique (tracial) Haar state (\cite{VD2}, Theorem 1.3), denoted by $\int_{\mathbb{G}}$, and whose invariance can be given by either of the equivalent conditions (\ref{invar2}) or (\ref{invar}).

.

\subsubsection{The Support of a State}\label{supports}
Let $\nu\in M_p(\mathbb{G})$ be a state and consider the left ideal
$$N_\nu=\{g\in F(\mathbb{G})\,|\,\nu(|g|^2)=0\}.$$
As a left ideal of a finite-dimensional $\mathrm{C}^\ast$-algebra,  $N_\nu$ must be of the form $F(\mathbb{G})q_\nu$ for $q_\nu$ a projection such that $gq_\nu=g$ for all $g\in N_\nu$ \cite{BLACK}. It is the case that for all $f\in F(\mathbb{G})$,
\begin{equation*}
\nu(q_\nu)=\nu(fq_\nu)=\nu(q_\nu f)=0.\label{eq2}
\end{equation*}
This implies in particular that $\nu(N_\nu)=\{0\}$. Define $p_\nu:=\mathds{1}_\mathbb{G}-q_\nu$. It is the case that for all $f\in F(\mathbb{G})$,
\begin{equation}\nu(f)=\nu(p_\nu f)=\nu(fp_\nu)=\nu(p_\nu fp_\nu),\label{eq1}\end{equation}
and that $\nu(p_\nu)=1$.  $p_\nu$ is the smallest projection such that $\nu(p_\nu)=1$. Call $p_\nu$ by the \emph{support projection} of $\nu$.

\subsection{Random Walks on Quantum Groups}
\begin{definition}
  A random walk on a  quantum group is given by a state, $\nu\in M_p(\mathbb{G})$.
\end{definition}
To study a random walk on a quantum group is to study its distribution after $k$ transitions: the semigroup of convolution powers, $(\nu^{\star k})_{k\geq 1}$, defined inductively through
$$\nu^{\star (k+1)}=(\nu\otimes \nu^{\star k})\circ \Delta.$$
Of central interest are random walks that are \emph{ergodic}:
\begin{definition}
  A random walk $\nu$ on a finite quantum group is said to be \emph{ergodic} if the convolution powers $(\nu^{\star k})_{k\geq 1}$ converge to the Haar state. In this context, denote the Haar state by $\pi$, call the Haar state by the \emph{random distribution}, and say the random walk \emph{converges to random}.
\end{definition}
The random walk $\nu$ is associated  with a stochastic operator $T_\nu:F(\mathbb{G})\rightarrow F(\mathbb{G})$ (see Section \ref{Stoch}), and this object plays a key role in the current work.
\subsection{The Dual of a Quantum Group}

Consider the space, $\widehat{F(\mathbb{G})}$, of linear functionals on $F(\mathbb{G})$ of the form
 $$g\mapsto \int_{\mathbb{G}} gf\,.\,\,\qquad (f,\,g\in F(\mathbb{G}))$$
As $F(\mathbb{G})$ is finite-dimensional, the continuous and algebraic duals coincide. Furthermore,  the Haar state is faithful and so
 $$\langle f,g\rangle:=\int_{\mathbb{G}} f^*g$$
  defines an inner product making $F(\mathbb{G})$ a Hilbert space. Via the Riesz Representation Theorem for Hilbert spaces, for every element  $\varphi\in F(\mathbb{G})^*$, there exists a density $f_{\varphi}^*\in F(\mathbb{G})$ such that:
  $$\varphi(g)=\langle f_{\varphi}^*,g\rangle=\int_\mathbb{G} f_{\varphi}g\,,\,\,\qquad (g\in F(\mathbb{G}))$$
  so that $F(\mathbb{G})^*=\widehat{F(\mathbb{G})}$. This space can be given the structure of an algebra of functions on a finite quantum group by employing the contravariant dual functor to $F(\mathbb{G})$ and its structure maps.  The finite quantum group formed in this way is called the \emph{dual} of the finite quantum group $\mathbb{G}$, and is denoted by $\widehat{\mathbb{G}}$.    Note that $M_p(\mathbb{G})$ is a subset of $F(\widehat{\mathbb{G}})$.  The $*$-involution on $F(\widehat{\mathbb{G}})$ is given by:

    $$\varphi^*(f)=\overline{\varphi(S(f)^*)}\,.\,\,\qquad (\varphi\in F(\widehat{\mathbb{G}}),\,f\in F(\mathbb{G}))$$

   A density $f_\nu\in F(\mathbb{G})$ defines a state $\nu\in M_p(\mathbb{G})$ if and only if $f_\nu$ is positive and $\int_\mathbb{G} f_\nu=1$. Denote the map $f_\nu\mapsto \nu$ by $\mathcal{F}$. The density of $\varepsilon$ is  $f_\varepsilon=\eta/\int_\mathbb{G} \eta$, while the density of the Haar state is just $f_{\int_\mathbb{G}}=\mathds{1}_\mathbb{G}$. In the sequel, unless specified otherwise, $f_\nu$ will denote the density of a state $\nu\in M_p(\mathbb{G})$.

\bigskip

\subsection{Group Algebras\label{Cocom}}
 Fixing for this section $G$ a classical group, and denoting $\mathbb{C}G=F(\widehat{G})$, the algebra structure of $F(\widehat{G})$ is the image of $G$, together with its structure maps, and group axiom commutative diagrams, under the free functor. Such algebras of functions are called \emph{cocommutative}, and satisfy $\Delta=\tau\circ\Delta$, where $\tau(f\otimes g)=g\otimes f$ is the flip map. Each subgroup $H\leq G$ gives a non-zero projection denoted $\chi_H\in 2^{\widehat{G}}$:
\begin{equation}\chi_H:=\frac{1}{|H|}\sum_{h\in H}\delta^h.\label{GAP}\end{equation}
Note that $\chi_{\{e\}}=\delta^e=\mathds{1}_{\widehat{G}}$, and that $\chi_H=\int_H\in M_p(G)\subset F(\widehat{G})$.

\bigskip

States on $F(\widehat{G})$ are given by positive definite functions $u\in M_p(\widehat{G})\subset F(G)$ (see Bekka, de la Harpe, and Valette (Proposition C.4.2, \cite{Bekka})). Furthermore, there is a bijective correspondence between positive definite functions and unitary representations on $G$ together with a  vector. In particular, for each positive definite function $u$ there exists a unitary representation $\rho:G\rightarrow \operatorname{GL}(H)$ and a  vector $\xi\in H$ such that
\begin{equation} u(s)=\langle\xi,\rho(s)\xi\rangle,\label{posdef}\end{equation}
and for each unitary representation $\rho$ and  vector $\xi$ (\ref{posdef}) defines a positive definite function on $G$. For $u$ to be a state, it is necessary that $u(e)=1$ and so $\langle \xi,\xi\rangle=1$; i.e. $\xi$ is a unit vector. Therefore probabilities on $\widehat{G}$ can be chosen by selecting a given representation and unit vector.

\bigskip

The comultiplication on $F(\widehat{G})$ being $\delta^s\mapsto \delta^s\otimes \delta^s$ implies that for a random walk on $\widehat{G}$ given by $u\in M_p(\widehat{G})$, the convolution powers are $(u^{\star k})_{k\geq 1}=(u^k)_{k\geq 1}$, the pointwise-multiplication powers. The Haar state is given by $\delta^e=:\int_{\widehat{G}}$, and so the random walk $u$ is ergodic if and only if $|u(s)|=1$ for $s=e$ only.

\bigskip

\section{Stochastic Operators\label{Stoch}}
\subsection{Definition and Properties}
Considered as a Markov chain with finite state space $G=\{s_1,\dots,s_{|G|}\}$, a  random walk on a classical group (\ref{rv}) \emph{driven} by $\nu\in M_p(G)$  has stochastic operator $T_\nu\in M_{|G|}(\mathbb{C})$:
$$[T_\nu]_{ij}=\mathbb{P}[\xi_{k+1}=s_i\,|\,\xi_k=s_j]=\nu({s_is_j^{-1}}).$$
Then $T_\nu$ is an operator on $F(G)$ equal to
$$T_\nu=(\nu\otimes I_{F(G)})\circ \Delta.$$
Thus given a random walk on a  quantum group $\mathbb{G}$ driven by $\nu\in M_p(\mathbb{G})$, define its \emph{stochastic operator} by the same formula.
Sometimes the notation $P_\nu$ is used for $(\nu\otimes I_{F(\mathbb{G})})\circ \Delta$, and, as in Franz and Gohm \cite{FG}, $T_\nu$ reserved for $(I_{F(\mathbb{G})}\otimes \nu)\circ \Delta$. This boils down to a choice between generalising a right-invariant walk (\ref{rv}), or a left-invariant walk:
$$\xi_k=\xi_{k-1}\zeta_{k}.$$
This current work is using the generalisation of a right-invariant walk, and so the stochastic operator $(\nu\otimes I_{F(\mathbb{G})})\circ \Delta$ is used, with the notation $T_\nu$ to avoid a clash in notation with $p_\nu$, the support projection of a state $\nu\in M_p(\mathbb{G})$.

\bigskip

In the usual way, via its transpose, $T_\nu$ gives an operator on $F(\widehat{\mathbb{G}})$, given by, for $\varphi\in F(\widehat{\mathbb{G}})$ and $f\in F(\mathbb{G})$:
$$T_\nu^t(\varphi)(f)=\varphi(T_\nu(f)).$$ In the sequel write $\varphi T_\nu$ for $T_\nu^t(\varphi)$.
\begin{proposition}\label{propertiesofP} Let $\mathbb{G}$ be a finite quantum group and $\mu,\,\nu\in M_p(\mathbb{G})$. Then the following hold:
\begin{enumerate}
\item[i.] $\mu T_\nu=\nu\star\mu$.
\item[ii.] $T_\nu^k=T_{\nu^{\star k}}$.
\item[iii.]  $\varepsilon T_\nu^k=\nu^{\star k}$.
\item[iv.] $T_\nu$ is unital and completely positive.
\item[v.] $M_p(\mathbb{G})$ is invariant under $T_\nu$.
\item[vi.] $\displaystyle \int_\mathbb{G}\circ T_\nu=\int_{\mathbb{G}}$\qquad $\bullet$
\end{enumerate}
\end{proposition}

To use the results of Fagnola and Pellicer  the stochastic operator must be a \emph{Schwarz Map}. All completely positive maps are Schwarz.

\subsection{Spectral Analysis}\label{Spectral}
Given a random walk on a finite quantum group, as $T_\nu$ is a linear operator on a finite-dimensional $\mathrm{C}^\ast$-algebra $F(\mathbb{G})$, the convergence of $(T_\nu^k)_{k\geq 1}$, and thus via Proposition \ref{propertiesofP} iii. of the convolution powers $(\nu^{\star k})_{k\geq 1}$, is determined by its spectrum, $\sigma(T_\nu)$. Thus much of the standard spectral analysis of Markov chain stochastic operators (see, for example, \cite{cecc}), applies in the quantum context. This analysis is often focussed on ergodic random walks, where $1\in\sigma(T_\nu)$ is multiplicity-free, and the only eigenvalue of modulus one. Following Evans and H{\o}egh-Krohn \cite{EVANS}, in a context more general than random walks on  finite quantum groups, Fagnola and Pellicer say a number of things about $\sigma(T_\nu)$:
   \begin{proposition}\label{mult}
   If $T_\nu$ is the stochastic operator of a random walk on a finite quantum group, then $\sigma(T_\nu)\subset \overline{\mathbb{D}}$. If $1\in\sigma(T_\nu)$ is multiplicity-free, then, for some $d\geq1$, $\sigma(T_\nu)\cap \mathbb{T}$ consists precisely of all the $d$th roots of unity $\bullet$
   \end{proposition}

\section{Irreducibility}\label{irred}
A random walk on a classical group $G$ is said to be reducible if there are group elements that the random walk can not visit. If there are no such elements, the random walk is said to be irreducible. If $S\subset G$ is any subset of $G$ not visited by the walk, the indicator function $\mathds{1}_S$ has the property that  $\nu^{\star k}(\mathds{1}_S)=0$ for all $k\in \mathbb{N}$. This motivates the following definition (related to \emph{degeneracy} of states: see Lemma 3.3, \cite{Simeng}):
\begin{definition}
  A random walk on a finite quantum group $\mathbb{G}$ given by $\nu\in M_p(\mathbb{G})$ is said to be \emph{reducible} if there exists a non-zero $q\in 2^\mathbb{G}$ such that $\nu^{\star k}(q)=0$ for all $k\in\mathbb{N}$. If there are no such non-zero projections, the random walk is said to be \emph{irreducible.}
\end{definition}

The conditions for a random walk on a classical group to be irreducible are rather straightforward: a random walk on a classical group is irreducible if and only if the support is not concentrated on a proper subgroup.

\bigskip

\begin{definition}\label{subgrpdef}
If $\mathbb{G}$ and $\mathbb{H}$ are finite quantum groups and $\pi:F(\mathbb{G}) \rightarrow F(\mathbb{H})$ is a surjective
unital $^*$-homomorphism such that
$$\Delta_{F(\mathbb{H})}\circ \pi=(\pi\otimes\pi)\circ \Delta_{F(\mathbb{G})},$$
then $\mathbb{H}$ is called a \emph{subgroup} of $\mathbb{G}$.
\end{definition}
In the classical case, \emph{any} non-empty subset $\Sigma\subseteq G$ generates a subgroup $\langle\Sigma \rangle\leq G$. The quantum generalisation of this statement is not true.
\subsection{Idempotent States}
Consider a random walk on a finite quantum group $\mathbb{G}$ given by $\nu\in M_p(\mathbb{G})$. If the convolution powers $(\nu^{\star k})_{k\geq 1}$ converge they converge to an idempotent, a state $\nu_\infty$ such that $\nu_{\infty}=\nu_{\infty}\star \nu_{\infty}$. The Kawada-It\^{o} Theorem implies that for classical groups, all idempotent states are integration against the uniform Haar measure on some subgroup \cite{Kaw}.

\subsubsection{Group-Like Projections}
The notion of a group-like projection in the algebra of functions on a quantum group was first introduced by Lanstad and Van Daele \cite{Land}.
\begin{definition}
  A non-zero $p\in 2^\mathbb{G}$ is called a group-like projection if
  $$\Delta(p)(\mathds{1}_\mathbb{G}\otimes p)=p\otimes p.$$
\end{definition}
It can be shown that $\varepsilon(p)=1$ and $S(p)=p$ \cite{Land}. It is not difficult to show that if $\mathbb{H}\leq \mathbb{G}$ is a subgroup, $\mathds{1}_\mathbb{H}\in 2^\mathbb{G}$ is a group-like projection. Franz and Skalski show that there is a one-to-one correspondence between idempotent states and group-like projections.  In particular, they prove:
\begin{proposition}(Cor. 4.2, \cite{idempotent})\label{4.2}
  Let $\phi\in M_p(\mathbb{G})$. The following are equivalent:
  \begin{enumerate}
    \item[i.] $\phi$ is idempotent
    \item[ii.] there exists a group-like projection $p\in 2^\mathbb{G}$ such that for all $f\in F(\mathbb{G})$:
    $$\phi(f)=\frac{1}{\int_\mathbb{G} p}\int_\mathbb{G}fp\qquad \bullet$$
  \end{enumerate}
\end{proposition}
In particular, if $\phi$ is an idempotent state, its density is $f_\phi=p/\int_\mathbb{G}p$ and
$$\phi=\mathcal{F}(f_\phi)=\frac{1}{\int_\mathbb{G}p}\mathcal{F}(p).$$

In the setting of locally compact quantum groups, Kasprzak and So{\l}tan \cite{Kasp} use the Gelfand philosophy to refer to the virtual object corresponding to a group-like projection as a \emph{quasi-subgroup}. This paper will take the same approach, associating to a group-like projection $p$ a quasi-subgroup $\mathbb{S}\subseteq \mathbb{G}$, and writing $p=:\mathds{1}_\mathbb{S}$, and the associated idempotent state by $\phi_\mathbb{S}$. Note also that a subgroup $\mathbb{H}\leq \mathbb{G}$ is a quasi-subgroup.

\begin{proposition}\label{suppdens}
If $\mathbb{S}\subset \mathbb{G}$ is a proper quasi-subgroup given by a group-like projection $\mathds{1}_\mathbb{S}$, the support of $\phi_\mathbb{S}$, $p_{\phi_\mathbb{S}}=\mathds{1}_\mathbb{S}.$
\end{proposition}
\begin{proof}
By Proposition \ref{4.2},
$$\phi_\mathbb{S}=\frac{1}{\int_\mathbb{G} \mathds{1}_\mathbb{S}}\mathcal{F}(\mathds{1}_\mathbb{S}),$$
is an idempotent state such that
$$\phi_\mathbb{S}(\mathds{1}_\mathbb{S})=\int_\mathbb{G} \mathds{1}_\mathbb{S}\frac{1}{\int_\mathbb{G} \mathds{1}_\mathbb{S}}\mathds{1}_\mathbb{S}=1,$$
as $\mathds{1}_\mathbb{S}\in 2^\mathbb{G}$.  Let $p_{\phi_\mathbb{S}}$ be the support projection of $\phi_\mathbb{S}$, so that $p_{\phi_\mathbb{S}}<\mathds{1}_\mathbb{S}$, and $p_{\phi_\mathbb{S}}\mathds{1}_\mathbb{S}=p_{\phi_\mathbb{S}}$. Consider

$$\int_\mathbb{G}(\mathds{1}_\mathbb{S}-p_{\phi_\mathbb{S}})=\int_\mathbb{G}(\mathds{1}_\mathbb{S}-p_{\phi_\mathbb{S}} \mathds{1}_\mathbb{S})=\int_\mathbb{G} \mathds{1}_\mathbb{S}-\int_\mathbb{G}p_{\phi_\mathbb{S}} \mathds{1}_\mathbb{S}.$$

Note

$$\int_\mathbb{G}p_{\phi_\mathbb{S}} \mathds{1}_\mathbb{S}=\int_\mathbb{G}\mathds{1}_\mathbb{S}\cdot \int_\mathbb{G}p_{\phi_\mathbb{S}} \frac{\mathds{1}_\mathbb{S}}{\int_\mathbb{G} \mathds{1}_\mathbb{S}}=\int_\mathbb{G} \mathds{1}_\mathbb{S}\cdot \phi_\mathbb{S} (p_{\phi_\mathbb{S}})= \int_\mathbb{G} \mathds{1}_\mathbb{S}\cdot 1 =\int_\mathbb{G}\mathds{1}_\mathbb{S},$$

so that $\int_\mathbb{G} (\mathds{1}_\mathbb{S}-p_{\phi_\mathbb{S}})=0$, and as $\int_\mathbb{G}$ is faithful, $\mathds{1}_\mathbb{S}-p_{\phi_\mathbb{S}}=0$, and so the group-like projection of an idempotent state is also its support $\bullet$
\end{proof}
This consideration, and Proposition \ref{4.2}, motivates:
\begin{definition}
A state $\nu\in M_p(\mathbb{G})$ is \emph{supported on a quasi-subgroup} $\mathbb{S}$ if, where $\mathds{1}_\mathbb{S}$ is the group-like projection associated with $\mathbb{S}$, $\nu(\mathds{1}_\mathbb{S})=1$.
\end{definition}

It will be seen shortly that there are quasi-subgroups that are not subgroups. The easiest way to see this is through the following theorem:
\begin{theorem}(Th. 4.5, \cite{idempotent})\label{subgroup}
Let $\mathbb{G}$ be a   quantum group and $\phi_\mathbb{S}\in M_p(\mathbb{G})$ an idempotent state with group-like projection $\mathds{1}_\mathbb{S}$. The following are equivalent:
\begin{enumerate}
  \item[i.] $\mathbb{S}\leq \mathbb{G}$ is a subgroup;
  \item[ii.] the null space $N_{\phi_\mathbb{S}}$ is a two-sided ideal of $F(\mathbb{G})$;
  \item[iii.] the null space $N_{\phi_\mathbb{S}}$ is a self-adjoint ideal of $F(\mathbb{G})$;
  \item[iv.] the null space $N_{\phi_\mathbb{S}}$ is an (antipode) $S$-invariant ideal of $F(\mathbb{G})$;
  \item[v.] the projection $\mathds{1}_\mathbb{\mathbb{S}}$ is central $\bullet$
\end{enumerate}
\end{theorem}
Thus, by (\ref{centralproj}) and v., given a group-like projection $\mathds{1}_\mathbb{S}$ in a concrete algebra of functions $F(\mathbb{G})\cong \bigoplus_{i}M_{n_i}(\mathbb{C})$,  $\mathds{1}_\mathbb{S}$ corresponds to  a subgroup if and only if it is a sum of full identity matrices.
\subsubsection{Pal's Idempotents}
The Kac--Paljutkin quantum group $\mathfrak{G}_0$ has an algebra of functions structure
$$F(\mathfrak{G}_0)=\mathbb{C}\oplus\mathbb{C}\oplus\mathbb{C}\oplus\mathbb{C} \oplus M_2(\mathbb{C}),$$
with basis elements $\eta$, $e_2$, $e_3$, $e_4$, and $E_{ij}$ for $1\leq i,j\leq 2$.  Pal \cite{Pal} determined that there are eight idempotent states $\{\phi_1,\dots,\phi_8\}\subset M_p(\mathfrak{G}_0)$ on the Kac--Paljutkin quantum group, six of these are non-trivial. Franz and Gohm \cite{FG} show that four of these six correspond to  subgroups (in fact classical groups), but $\phi_6$ and $\phi_7$ are not. By looking at their supports:
\begin{align*}
  p_{\phi_6} & =\eta+e_4+E_{11}, \\
   p_{\phi_7}& =\eta+e_4+E_{22};
\end{align*}
it is easy to see that they correspond to quasi-subgroups that are not subgroups.

\subsubsection{Cocommutative Idempotents}\label{CCID}
Pal's counterexample showed, as Franz and Skalski remark \cite{ergodic}, that the necessary and sufficient conditions for ergodicity of a random walk on a quantum group are ``clearly more complicated'' (than the classical situation). In fact, as was noted after Pal's counterexample, there exists an abundance of quasi-subgroups that are not subgroups as soon as cocommutative algebras of functions are considered.

\bigskip

Every subset $S\subseteq G$ gives an indicator function $\mathds{1}_S$ that is an idempotent in $F(G)$. However not all of these are positive definite functions. Firstly if $\mathds{1}_S$ is to be a state, the subset $S$ must be a subgroup (Ex. 6.C.4, \cite{Bekka}). Consider a cocommutative algebra of functions $F(\widehat{G})$ and $H\leq G$. The indicator function on $H$, $\mathds{1}_H$, is an idempotent state on $\widehat{G}$. Its density is
$$\varphi_{\mathds{1}_H}=\sum_{h\in H}\delta^h=\frac{\chi_H}{\int_{\widehat{G}}\chi_H}.$$
Therefore, by Proposition \ref{4.2} its associated group-like projection is equal to $\chi_H$, and by Proposition \ref{suppdens} the support projection $p_{\mathds{1}_H}=\chi_H\in F(\widehat{G})$.

\bigskip

By Theorem \ref{subgroup} v., the quasi-subgroup given by $\chi_H$ is a subgroup if and only if  $\chi_H$ is central, that is for all $s\in G$:
$$\chi_H\delta^s=\frac{1}{|H|}\sum_{h\in H}\delta^{hs}=\frac{1}{|H|}\sum_{h\in H}\delta^{sh}=\delta^s\chi_H,$$
which is the case if and only if $H$ is a normal subgroup of $G$.

\bigskip

Therefore whenever $H\leq G$ is a non-normal subgroup, $\chi_H$ gives a quasi-subgroup of $\widehat{G}$ which is not a subgroup.

\subsection{Harmonic Projections}
Fagnola and Pellicer  identify  hereditary subalgebras as the appropriate quantum generalisation of functions concentrated on subsets, and this motivates their definition of irreducibility:
\begin{definition}
A stochastic operator is \emph{irreducible in the sense of Fagnola and Pellicer} if there exists no proper hereditary $T_\nu$-invariant $\mathrm{C}^\ast$-subalgebras of $F(\mathbb{G})$. A $p\in 2^\mathbb{G}$ is called $T_\nu$\emph{-subharmonic} if $T_\nu(p)\geq p$.
\end{definition}
 Fagnola and Pellicer prove:
\begin{theorem}\label{prev}
A stochastic operator $T_\nu$ is irreducible in the sense of Fagnola and Pellicer if and only if the only $T_\nu$-subharmonic projections are the trivial $0$ or $\mathds{1}_\mathbb{G}$ $\bullet$
\end{theorem}
Fagnola and Pellicer use the terminology of $T_\nu$-subharmonic, but the operator norm $\|T_\nu\|=1$, and so $T_\nu$-harmonic is used here. As expected, irreducible in the sense of Fagnola and Pellicer coincides with the definition of irreducible for random walks.
\begin{theorem}\label{FPJ}
  A stochastic operator $T_\nu$ is irreducible in the sense of Fagnola and Pellicer if and only if the associated random walk is irreducible.
\end{theorem}
\begin{proof}
Assume that $T_\nu$ is irreducible in the sense of Fagnola and Pellicer.
  As $F(\mathbb{G})$ is a finite-dimensional $\mathrm{C}^\ast$-algebra, it may be concretely realised as
\begin{equation}\label{exp}F(\mathbb{G})\cong \bigoplus_{i=1}^N M_{n_i}(\mathbb{C})\subset B(\mathbb{C}^{\dim H}),\end{equation}
the bounded operators on a Hilbert space of dimension $\displaystyle \dim H=\sum_{i=1}^{N}n_i$. An inner product is given by:
$$\langle f,g\rangle=\int_\mathbb{G}f^*g.$$
Let $q\in 2^\mathbb{G}$ and suppose $T_\nu(q)=0$. By Proposition \ref{propertiesofP} iv., $T_\nu(\mathds{1}_{\mathbb{G}}-q)=\mathds{1}_{\mathbb{G}}\geq \mathds{1}_{\mathbb{G}}-q$, so that $\mathds{1}_{\mathbb{G}}-q$ is $T_\nu$-subharmonic. By Theorem \ref{prev}, $\mathds{1}_{\mathbb{G}}-q=\mathds{1}_{\mathbb{G}}$, in which case $q=0$, and there will be nothing to say; or $\mathds{1}_{\mathbb{G}}-q=0$, in which case $q=\mathds{1}_{\mathbb{G}}$ and $T_\nu(q)=\mathds{1}_{\mathbb{G}}\neq 0$.

\bigskip

Therefore assume $T_\nu(q)\neq 0$. If $\nu(q)>0$, there is nothing to say, so assume $\nu(q)=0\Rightarrow \varepsilon(T_\nu(q))=0$. This implies that, where $\eta=\eta^*$ is the Haar element:
$$\int_\mathbb{G}\eta \,T_\nu(q)=0\Rightarrow\langle\eta,T_\nu(q)\rangle=0.$$
As a positive linear map on $B(H)$, $T_\nu$ satisfies the hypothesis of Proposition 2.2 of Evans-H{\o}egh-Krohn \cite{EVANS}. Therefore there exists a $k<\dim H$ such that
$$\langle\eta,(T_\nu)^k(T_\nu(q))\rangle>0\Rightarrow \int_\mathbb{G} \eta\, T_{\nu^{\star (k+1)}}(q)>0.$$
Therefore
\begin{align*}
  \varepsilon(T_{\nu^{\star (k+1)}}(q))  >0\Rightarrow \nu^{\star (k+1)}(q) & >0,
\end{align*}
and so the random walk given by $\nu$ is irreducible.

\bigskip

Suppose now that $T_\nu$ is reducible in the sense of Fagnola and Pellicer, so that there exists a non-trivial $T_\nu$-harmonic $q$ such that $T_\nu(q)=q$ and indeed $T_\nu^k(q)=q$ for all $k\in \mathbb{N}$. This implies that for all $k\in\mathbb{N}$
$$\nu^{\star k}(q)=\varepsilon(T_\nu^k(q))=\varepsilon(q).$$
Where $\eta\in 2^\mathbb{G}$ is the Haar element, if
    \begin{align*}
    q&=\alpha_q\eta\oplus r,
  \end{align*}
  $\alpha_q$ is zero or one. If $\alpha_q=0$ then
  $$\nu^{\star k}(q)=\varepsilon(q)=0$$
  for all $k\in\mathbb{N}$, and so the random walk given by $\nu$ is reducible. If $\alpha_q=1$, then $p:=\mathds{1}_\mathbb{G}-q$ is a non-zero projection such that $\nu^{\star k}(p)=0$ for all $k\in \mathbb{N}$, so that the random walk given by $\nu$ is reducible $\bullet$
\end{proof}
It can be seen in the proof of Theorem \ref{FPJ}, via Evans-H{\o}egh-Krohn, that the $k_0$ referenced below can be taken to be the dimension of the Hilbert space upon which $F(\mathbb{G})$ is the set of bounded operators:
 \begin{corollary}\label{CAY}
   Suppose that $\nu$ is an irreducible random walk on a quantum group. Then there exists $k_0\in\mathbb{N}$ such that for all non-zero $q\in 2^\mathbb{G}$, there exists $k\leq k_0$ such that $\nu^{\star k}(q)>0$.
 \end{corollary}

\subsection{Irreducibility Criterion}

\begin{proposition}\label{quassup}
Let $\nu,\,\mu\in M_p(\mathbb{G})$ be supported on a quasi-subgroup $\mathbb{S}$. Then $\nu\star\mu$ is also supported on $\mathbb{S}$.
\end{proposition}
\begin{proof}
That $\mathds{1}_\mathbb{S}$ is a group-like projection implies that (using Sweedler notation)
\begin{align*}
  \Delta(\mathds{1}_\mathbb{S})(\mathds{1}_\mathbb{G}\otimes \mathds{1}_\mathbb{S}) & =\mathds{1}_\mathbb{S}\otimes \mathds{1}_\mathbb{S} \\
  \Rightarrow \sum \mathds{1}_{\mathbb{S}(1)}\otimes (\mathds{1}_{\mathbb{S}(2)}\mathds{1}_\mathbb{S}) & =\mathds{1}_\mathbb{S}\otimes \mathds{1}_\mathbb{S}.
\end{align*}
Hit both sides with $\nu\otimes \mu$:
$$\sum \nu(\mathds{1}_{\mathbb{S}(1)})\mu(\mathds{1}_{\mathbb{S}(2)}\mathds{1}_\mathbb{S})=\nu(\mathds{1}_\mathbb{S})\mu(\mathds{1}_\mathbb{S})=1,$$
as $\nu,\,\mu$ are supported on $\mathbb{S}$. Note that $r_\mu:=\mathds{1}_\mathbb{S}-p_\mu\in N_\mu$ as $\mu(r_\mu^*r_\mu)=0$. Furthermore, as $N_\mu$ is a left ideal, $\mathds{1}_{\mathbb{S}(2)}r_\mu\in N_\mu$. Now consider
\begin{align*}
  \mu(\mathds{1}_{\mathbb{S}(2)}\mathds{1}_\mathbb{S})=\mu(\mathds{1}_{\mathbb{S}(2)}(p_\mu+r_\mu)) & =\mu(\mathds{1}_{\mathbb{S}(2)}p_\mu)+\mu(\mathds{1}_{\mathbb{S}(2)}r_\mu) \\
   & =\mu(\mathds{1}_{\mathbb{S}(2)}),
\end{align*}
as $\mu(N_\mu)=\{0\}$ and by (\ref{eq1}). This means that
$$\sum \nu(\mathds{1}_{\mathbb{S}(1)})\mu(\mathds{1}_{\mathbb{S}(2)})=1.$$
However this is the same as
$$(\nu\otimes\mu)\Delta(\mathds{1}_{\mathbb{S}})=1\Rightarrow (\nu\star \mu)(\mathds{1}_{\mathbb{S}})=1\qquad\bullet$$
\end{proof}

\begin{theorem}\label{irrcond}
A random walk $\nu$ is irreducible if and only if $\nu$ is not supported on a proper quasi-subgroup.
\end{theorem}
\begin{proof}
Suppose that $\nu $ is supported on a proper quasi-subgroup $\mathbb{S}$, so that $p_\nu\leq \mathds{1}_\mathbb{S}<\mathds{1}_\mathbb{G}$. By Proposition \ref{quassup}, for all $k\in\mathbb{N}$,  $\nu^{\star k}$ is supported on $\mathbb{S}$. Consider the projection $q_\mathbb{S}:=\mathds{1}_\mathbb{G}-\mathds{1}_\mathbb{S}>0$. Then for all $k\in\mathbb{N}$
$$\nu^{\star k}(\mathds{1}_\mathbb{S})=1\Rightarrow \nu^{\star k}(\mathds{1}_\mathbb{G}-q_\mathbb{S})=1\Rightarrow \nu(q_\mathbb{S})=0,$$
that is the random walk given by $\nu$ is reducible.

\bigskip

Suppose now that the random walk given by $\nu$ is reducible so that there is a non-zero  $q\in 2^\mathbb{G}$ such that for all $k\in\mathbb{N}$, $\nu^{\star k}(q)=0$. This implies that for all $n\in\mathbb{N}$, $\nu_n(q)=0$, where
$$\nu_n:=\frac{1}{n}\sum_{k=1}^n\nu^{\star k}.$$
Where
$$\nu_\infty:=\lim_{n\rightarrow \infty}\nu_n,$$
$\nu_\infty$ is an idempotent state (this is well known, see e.g. Th. 7.1, \cite{FG}) such that $\nu_\infty(q)=0$. Thus $\nu_\infty$ cannot be the Haar state as the Haar state is faithful.

\bigskip

Where $p_{\nu_n}$ is the support projection of $\nu_n$,
$$\nu_n(p_{\nu_n})=\frac{1}{n}\sum_{k=1}^n\nu^{\star k}(p_{\nu_n})=1\Rightarrow \nu^{\star k}(p_{\nu_n})=1,$$
for each $1\leq k\leq n$. Now consider $\nu_m$ with $m<n$. As $\nu^{\star k}(p_{\nu_n})=1$ for all $1\leq k\leq m<n$, $\nu_m(p_{\nu_n})=1$ and so $p_{\nu_m}\leq p_{\nu_n}$.  Note that $F(\mathbb{G})\cong B(H)$, and so $(p_{\nu_n})_{n\geq 1}$ is an increasing net of projections on some Hilbert space, and therefore (Th. 4.1.2, \cite{Murph}) $p_{\nu_\infty}$ is the projection  onto the closed vector subspace:$$\overline{\lim_{N\rightarrow \infty}\bigcup_{n=1}^{N}p_{\nu_n}(H)}.$$
As  $p_{\nu_n}(H)\subseteq p_{\nu_\infty}(H)$, each $p_{\nu_n}\leq p_{\nu_{\infty}}$ (Th. 2.3.2, \cite{Murph}), and this implies that:
\begin{equation}
p_\nu=p_{\nu_1}\leq p_{\nu_2}\leq p_{\nu_3}\leq \cdots \leq p_{\nu_n}\leq \cdots \leq p_{\nu_\infty}< p_{\int_\mathbb{G}}=\mathds{1}_{\mathbb{G}},\label{chain}\end{equation}
in particular $\nu$ is concentrated on the quasi-subgroup given by $p_{\nu_\infty}$ $\bullet$\end{proof}

\section{Periodicity}
If a random walk is irreducible, the other way it can fail to be ergodic is if periodic behavior occurs. In the classical case, if a random walk given by $\nu\in M_p(G)$ is irreducible, yet fails to be ergodic, one can construct a proper normal subgroup $N\lhd G$, and show that $\operatorname{supp }\nu\subseteq gN$. As the random walk is irreducible, it must be the case that  $G/N\cong \mathbb{Z}_{d}$, where $d:=[G:N]$.

\bigskip

 Fagnola and Pellicer define:
\begin{definition}
  Let $T_\nu$ be the stochastic operator of an irreducible random walk. A partition of unity $\{p_i\}_{i=0}^{d-1}\subset F(\mathbb{G})$ is called $T_\nu$-cyclic if (where the subtraction is understood mod $d$):
  $$T_\nu(p_iF(\mathbb{G})p_i)=p_{i-1}F(\mathbb{G})p_{i-1}.$$
  The stochastic operator, and the associated random walk, is called \emph{periodic} if there exists a $T_\nu$-cyclic partition of unity with $d\geq 2$. The biggest such $d$ is called the \emph{period} of the random walk.
\end{definition}
Proposition 4.1 of Fagnola and Pellicer states that $\{p_i\}_{i=0}^{d-1}$ is $T_\nu$-cyclic if and only if $T_\nu(p_i)=p_{i-1}$. Furthermore, Theorems 3.7 and 4.3 of Fagnola and Pellicer imply:
\begin{proposition}\label{cyclic}
If $\nu$ is an irreducible but periodic random walk, there exists a $T_\nu$-cyclic partition of unity, $\{p_i\}_{i=0}^{d-1}$ such that
$$T_\nu(p_i)=p_{i-1},$$
where the subtraction is understood mod $d$ $\bullet$
\end{proposition}
Clearly each $p_i$ is harmonic for $T_{\nu^{\star d}}$. These $T_\nu$-cyclic partitions of unity behave very much like indicator functions of cosets of normal  subgroups of classical groups, such that $\nu$ is concentrated on the coset, given by $p_1$, of the normal subgroup given by $p_0$.
\begin{proposition}\label{propncyclic}
Suppose that $\{p_i\}_{i=0}^{d-1}$ is a $T_\nu$-cyclic partition of unity. Then the indexing $i=0,1,\dots,d-1$ can be chosen such that
\begin{enumerate}
\item[i.] $$\varepsilon(p_i)=\begin{cases}
                       1, & \mbox{if } i=0, \\
                       0, & \mbox{otherwise}.
                     \end{cases}$$

  \item[ii.]$$\nu(p_i)=\begin{cases}
             1, & \mbox{if } i=1, \\
             0, & \mbox{otherwise}.
           \end{cases}$$
\end{enumerate}
Furthermore $\displaystyle\int_\mathbb{G}p_i=\frac{1}{d}$.
\end{proposition}

\begin{proof}
\begin{enumerate}
\item[i.] Where $\eta\in 2^\mathbb{G}$ is the \emph{Haar element}, writing
    \begin{align*}
    p_i&=\alpha_i\eta\oplus r_i,
  \end{align*}
  each $\alpha_i$ is zero or one. Note
  $$\sum_{i=0}^{d-1}p_i=\left(\sum_{i=0}^{d-1}\alpha_i\right)\eta\oplus\left(\sum_{i=0}^{d-1}r_i\right) =\mathds{1}_\mathbb{G}=1\eta\oplus \bigoplus_{i}I_{n_i},$$
  and this implies that only one of the $\alpha_i=1$. Choose it to be $i=0$.
  \item[ii.] From Proposition \ref{propertiesofP} iii., $\nu(p_i)=\varepsilon(p_{i-1})$.
  \end{enumerate}

  By Proposition \ref{propertiesofP} v.,
  $$\int_\mathbb{G}\circ \,T_{\nu^{\star k}}=\int_\mathbb{G}.$$
  Let $i,j\in\{0,1,\dots,d-1\}$:
  \begin{align*}
    \int_\mathbb{G} T_{\nu^{\star (i-j)}}(p_i)  =\int_\mathbb{G} p_i    \Rightarrow \int_\mathbb{G} p_j  =\int_\mathbb{G} p_i=:c,
  \end{align*}
and so for any $j\in\{0,1,\dots,d-1\}$,
  $$dc=d\cdot \int_\mathbb{G} p_j=\sum_{i=0}^{d-1}\int_\mathbb{G} p_i=\int_\mathbb{G}\sum_{i=0}^{d-1}p_i=\int_\mathbb{G} \mathds{1}_\mathbb{G}=1\qquad \bullet$$
\end{proof}
For the remainder of the current work,  this indexing is understood.

\begin{theorem}\label{GLPT}
Suppose that $\{p_i\}_{i=0}^{d-1}$ is a $T_\nu$-cyclic partition of unity for an irreducible random walk. Then $p_0$ is a group-like projection.
\end{theorem}
\begin{proof}
If $d=1$, then $p_0=\mathds{1}_\mathbb{G}$ is a group-like projection. Therefore assume $d>1$. Using the Pierce decomposition with respect to $p_0$, where $q_0=\mathds{1}_\mathbb{G}-p_0$,
$$F(\mathbb{G})=p_0 F(\mathbb{G}) p_0+p_0F(\mathbb{G})q_0+q_0F(\mathbb{G})p_0+q_0F(\mathbb{G})q_0.$$
 As $\nu$ is irreducible, by Corollary \ref{CAY}, there exists a $k_0\in\mathbb{N}$, such that for all non-zero $q\in 2^\mathbb{G}$, there exists $k_q\leq k_0\in \mathbb{N}$ such that $\nu^{\star k_q}(q)>0$.

 \bigskip

 Let $\phi:=\nu^{\star d}$ so that, via $T_{\phi}=T_\nu^d$, $T_\phi(p_0)=p_0$ and $T_\phi(p_0F(\mathbb{G})p_0)=p_0F(\mathbb{G})p_0$. Define:

$$\phi_n=\frac{1}{n}\sum_{k=1}^n\phi^{\star k}.$$
Consider $\phi^{\star k}(p_0)$ for any $k\in\mathbb{N}$. Note that

\begin{equation}\label{supp}
  \phi^{\star k}(p_0)=\varepsilon (T_{\phi^{\star k}}(p_0))=\varepsilon (T^k_\phi(p_0))=\varepsilon (T^k_{\nu^{\star d}}(p_0))=\varepsilon(T_\nu^{kd}(p_0))=\varepsilon(p_0)=1,
\end{equation}
that is each $\phi^{\star k}$ is supported on $p_0$. This means furthermore that $\phi_{k_0}(p_0)=1$. The corner $p_0F(\mathbb{G})p_0$ is a hereditary $\mathrm{C}^\ast$-subalgebra, such that $p_0\in p_0F(\mathbb{G})p_0$. Suppose that the support $p_{\phi_{k_0}}<p_0$. This implies that $p_{\phi_{k_0}}\in p_0F(\mathbb{G})p_0$ (Sec. 3.2, \cite{Murph}).

\bigskip

Consider the projection $r:=p_0-p_{\phi_{k_0}}\in p_0F(\mathbb{G})p_0$. There exists a $k_r\leq k_0$ such that

$$0<\nu^{\star k_r}(p_0-p_{\phi_{k_0}})\Rightarrow \nu^{\star k_r}(p_{\phi_{k_0}})<\nu^{\star k_r}(p_0).$$

This implies that $\nu^{\star k_r}(p_0)>0\Rightarrow k_r\equiv 0\mod d$, say $k_r=\ell_r\cdot d$ (note $\ell_r\leq k_0$):

\begin{align*}\nu^{\star \ell_r\cdot d}(p_{\phi_{k_0}})&<\nu^{\star \ell_r\cdot d}(p_0)\\ \Rightarrow (\nu^{\star d})^{\star \ell_r}(p_{\phi_{k_0}})&<(\nu^{\star d})^{\star \ell_r}(p_0)\\ \Rightarrow \phi^{\star \ell_r}(p_{\phi_{k_0}})&<\phi^{\star \ell_r}(p_0)\\ \Rightarrow \phi^{\star \ell_r}(p_{\phi_{k_0}})&<1\end{align*}

By assumption $\phi_{k_0}(p_{\phi_{k_0}})=1$. Consider

$$\phi_{k_0}(p_{\phi_{k_0}})=\frac{1}{k_0} \sum_{k=1}^{k_0}\phi^{\star k}(p_{\phi_{k_0}}).$$

For this to equal one every $\phi^{\star k}(p_{\phi_{k_0}})$ must equal one for $k\leq k_0$, but $\phi^{\star \ell_r}(p_{\phi_{k_0}})<1$. Therefore $p_0$ is the support of $\phi_{k_0}$.

\bigskip

Define
$$\phi_\infty=\lim_{n\rightarrow\infty} \phi_n.$$
This is an idempotent state. Consider (\ref{chain}) for $\phi$, but note by (\ref{supp}) that $p_{\phi_\infty}\leq p_0$:
$$p_\phi=p_{\phi_1}\leq \cdots \leq p_{\phi_{k_0}}\leq \cdots \leq p_{\phi_\infty}\leq p_0,$$
however $p_{\phi_{k_0}}=p_0$ which squeezes $p_{0}=p_{\phi_{\infty}}$, so $p_0$ is the support of a group-like projection, and therefore, by Proposition \ref{suppdens}, $p_0$ is a group-like projection $\bullet$
\end{proof}

The possibility remains that $p_0$ might always correspond to a subgroup. The following example shows that this is not the case.
\subsection{Cocommutative Example} Consider the algebra of functions on a dual group $\widehat{G}$. If $H\leq G$ is a subgroup,
$$\chi_H=\frac{1}{|H|}\sum_{h\in H}\delta^h$$
is a group-like projection, and so corresponds to a quasi-subgroup. The quasi-subgroup is a subgroup if and only if $H\lhd G$.

\bigskip

Consider the algebra of functions on the dual group $\widehat{S_3}$, and a state $u\in M_p(\widehat{S_3})$ given by:
$$u(\sigma)=\langle\xi,\rho(\sigma)\xi\rangle,$$
where $\rho$ is the permutation representation $S_3\rightarrow \operatorname{GL}(\mathbb{C}^3)$, $\rho(\sigma)e_i=e_{\sigma(i)}$, and $$\xi=\left(\frac{1}{\sqrt{2}},-\frac{1}{\sqrt{2}},0\right)\in \mathbb{C}^3.$$
Indeed
$$u(\delta^\sigma)=\begin{cases}
    1, & \mbox{if } \sigma=e \\
    -1, & \mbox{if } \sigma=(12) \\
    -\frac12\operatorname{sgn}(\sigma), & \mbox{otherwise}.
  \end{cases}$$
Let $p=\sum_{\sigma\in S_3}\alpha_\sigma\delta^\sigma\in F(\widehat{S_3})$ be a fixed point of $T_u$:
\begin{align*}
T_u(p)=(u\otimes I_{F(\widehat{S_3})})\circ \Delta(p)= \sum_{\sigma\in S_3}\alpha_\sigma u(\delta^\sigma)\delta^\sigma.
\end{align*}
This implies that either $p=0$ or $p=\mathds{1}_{\widehat{S_3}}$. This implies that $T_u$ is irreducible in the sense of Fagnola and Pellicer, and thus irreducible (Th. \ref{prev} and \ref{FPJ}).

\bigskip

Define $p_0:=\chi_{\langle (12)\rangle}$ and $p_1=\mathds{1}_{\widehat{S_3}}-p_0$. Note that $(p_0,p_1)$ is a $T_u$-cyclic partition of unity, but $p_0$ does not correspond to subgroup of $\widehat{S_3}$ because $\langle (12)\rangle$ is not normal in $S_3$.

\begin{definition}
Let $\mathbb{G}$ be a finite quantum group. A state $\nu\in M_p(\mathbb{G})$ is \emph{supported on a cyclic coset of a proper quasi-subgroup} if there exists a pair of projections $p_0,\,p_1\in 2^{\mathbb{G}}$, such that $p_0p_1=0$, $p_0+p_1\leq \mathds{1}_{\mathbb{G}}$, $\nu(p_1)=1$, $p_0$ is a group-like projection, $T_\nu(p_1)=p_0$, and there exists $d>1$ such that $T_\nu^d(p_1)=p_1$.
\end{definition}
Considering the classical Markov chain theory that ergodic is equivalent to irreducible and aperiodic, it is easy to infer in the classical case that irreducible is equivalent to not concentrated on a subgroup, and, independently, aperiodic is equivalent to not concentrated on a coset of a normal subgroup. This however is incorrect: assuming irreducibility it is true that aperiodic is equivalent to not concentrated on a coset of a proper normal subgroup, but  a reducible random walk on a finite group $G$ can be periodic without being concentrated on any coset of a proper normal subgroup of $G$: consider a random walk concentrated on a coset of a subgroup $N$ normal in a subgroup $H$ of $G$ but $N$ is not normal in $G$. What periodic \emph{is} equivalent to is that the driving probability is concentrated on a coset of a proper normal subgroup $N\lhd H$ of a subgroup $H\leq G$. The definition of a cyclic coset of a proper quasi-subgroup is capturing this latter case: classically $p_0,p_1$ are indicator functions $\mathds{1}_N,\,\mathds{1}_{Ng}$ for some $N\lhd H$, $H\leq G$, such that, for some $d>1$, $H/N\cong \mathbb{Z}_d$. However as soon as a random walk is irreducible and periodic, $N\lhd G$, and the $p_0,\,p_1$ are elements of a full partition of $\mathds{1}_G$.

\subsection{Ergodic Theorem}
The main result may now be stated:
\begin{theorem}(The Ergodic Theorem for Random Walks on Finite Quantum Groups)\label{Erg}
A random walk on a finite quantum group $\mathbb{G}$ given by $\nu\in M_p(\mathbb{G})$ is ergodic if and only if the state is not supported on a proper quasi-subgroup, nor on a cyclic coset of a proper quasi-subgroup.
\end{theorem}
\begin{proof}
Assume that the support of $\nu$, $p_\nu\leq \mathds{1}_\mathbb{S}< \mathds{1}_\mathbb{G}$ for a proper quasi-subgroup $\mathbb{S}\subset \mathbb{G}$. By Proposition \ref{quassup}, $p_{\nu^{\star k}}\leq p_\mathbb{S}$ for all $k\in\mathbb{N}$, and thus for $q_\mathbb{S}:=\mathds{1}_\mathbb{G}-\mathds{1}_\mathbb{S}>0$, $\nu^{\star k}(q_\mathbb{S})=0$ for $k\in\mathbb{N}$ and so $\nu$ is not ergodic. Assume now that the support of $\nu$ is concentrated on a cyclic coset of a proper quasi-subgroup. Similar arguments to those used in Proposition \ref{propncyclic} ii. show that $\int_{\mathbb{G}}p_1=\int_{\mathbb{G}}p_0$. Note that $p_0+p_1\leq \mathds{1}_{\mathbb{G}}$ so that
$$\int_{\mathbb{G}}(p_0+p_1)=2\int_{\mathbb{G}}p_0\leq 1\Rightarrow \int_{\mathbb{G}}p_0\leq \frac{1}{2}.$$
Choose $d>1$ such that $T_\nu^d(p_1)=p_1$.
 If $\nu$ were ergodic,
$$\lim_{k\rightarrow \infty}\nu^{\star (dk+1)}=\int_\mathbb{G}\Rightarrow \lim_{k\rightarrow \infty}\nu^{\star(dk+1)}(p_1)=\int_\mathbb{G} p_0\leq \frac{1}{2}.$$
However for all $k\in\mathbb{N}$:
$$\nu^{\star(dk+1)}(p_1)=\varepsilon(T_\nu^{dk+1}(p_1))=\varepsilon(T_\nu^{dk}(p_0))=1,$$
is constant not equal to $\int_\mathbb{G} p_0$, and so $\nu$ is not ergodic.

\bigskip

Assume now that $\nu$ is not ergodic. If $\nu$ is reducible, by Proposition \ref{irrcond}, $\nu$ is concentrated on the proper quasi-subgroup given by $p_{\nu_\infty}$. Assume therefore that $\nu$ is irreducible: by Theorem \ref{prev}, $1\in\sigma(T_\nu)$ is multiplicity free, and so by Proposition \ref{mult}, $\sigma(T_\nu)\cap \mathbb{T}$ consists  of all $d$th roots of unity.  If $\nu$ is not periodic, then $d=1$, $\sigma(T_\nu)\cap \mathbb{T}=\{1\}$, so that $T_\nu^k$ converges, and so $\varepsilon T_\nu^k= \nu^{\star k}$ converges to an idempotent state $\phi$. This idempotent is necessarily such that $\nu\star\phi=\phi=\phi\star \nu$, and so by Proposition \ref{propertiesofP} i., $\phi T_\nu=\nu\star\phi=\phi$. However the multiplicity of $1\in\sigma(T_\nu)$ is one, and so $\nu^{\star k}\rightarrow \phi=\int_{\mathbb{G}}$, that is irreducible and not periodic implies ergodic.  Therefore assume that $\nu$ is irreducible but periodic. Proposition \ref{cyclic} provides a $T_\nu$-cyclic partition of unity $\{p_{i}\}_{i=0}^{d-1}$ such that $d>1$, so that $p_0p_1=0$, $p_0+p_1\leq \mathds{1}_G$, and $T_\nu(p_1)=p_0$. Note that
$$\nu(p_1)=\varepsilon(T_\nu(p_1))=\varepsilon(p_0)=1,$$
so that the support of $\nu$, $p_\nu\leq p_1$. By Theorem \ref{GLPT}, $p_0$ is a group-like projection. Finally $T_\nu^d(p_1)=p_1$ so that $\nu$ is supported on a cyclic coset of a proper quasi-subgroup $\bullet$
\end{proof}
As an easy corollary, a finite version of a result of Franz and Skalski:
\begin{corollary}(Prop. 2.4, \cite{ergodic})
A random walk on a finite quantum group given by a faithful $\nu\in M_p(\mathbb{G})$ is ergodic.
\end{corollary}
\begin{proof}
The support $p_\nu=\mathds{1}_\mathbb{G}$ $\bullet$.
\end{proof}
\subsection{Discussion}
A new proof of the classical theorem follows. The necessary conditions are easy.  Suppose that a random walk on a classical group given by $\nu\in M_p(G)$ is not ergodic. In the classical case, by Theorem \ref{subgroup} v., all quasi-subgroups are subgroups, and $p_\nu\leq \mathds{1}_H$ for $H<G$ a proper subgroup.

\bigskip

Suppose that the random walk is irreducible. Then Theorem 4.3 of Fagnola and Pellicer provides a $T_\nu$-cyclic partition of unity $\{p_i\}_{i=0}^{d-1}$,  and thus a partition $G=\biguplus_{i=0}^{d-1} S_i$, with $p_i=\mathds{1}_{S_i}$. By definition $S_0=N<G$ a proper subgroup, and $\nu(S_1)=1$. Using the random variable picture (\ref{rv}), each $\zeta_i\in S_1$, and thus $\xi_k\in S_1^k$. As the walk is not concentrated on a subgroup,  and thus every $s\in G$ is in some $S_1^k=S_{k}$, where $S_k$ is understood mod $d$.

\bigskip

 Define a map $\theta:G\rightarrow \mathbb{Z}_d$ by $\{S_i\}\rightarrow \{i\}\subset \mathbb{Z}_d$. Elements $s_i\in S_i=S_1^i$ and $s_j\in S_j=S_1^j$ satisfy $s_is_j\in S_1^iS_1^j=S_1^{i+j}=S_{i+j}$, and thus $\theta$ is a homomorphism, and its kernel, $S_0=N$, is a proper normal subgroup $N\lhd G$, and so $\theta^{-1}(1)=S_1$ is a coset of a proper normal subgroup.

\bigskip

Immediately an issue in the quantum case is that $p_0=\mathds{1}_\mathbb{S}$ is only a quasi-subgroup $\mathbb{S}\subset \mathbb{G}$, and not a subgroup. It would appear fruitful to consider `cosets of quasi-subgroups', perhaps using the notion of shifts of group-like projections, \cite{Kasp2}.  However, even if, as  could be conjectured, that for a $T_\nu$-cyclic partition of unity $\{p_i\}_{i=0}^{d-1}$:
$$\Delta(p_i)=\sum_{j=0}^{d-1}p_{i-j}\otimes p_j,$$
and some class of quotient of $\mathbb{G}$ by the quasi-subgroup $\mathbb{S}$  be constructed, such that $\text{``}F(\mathbb{G}/\mathbb{S})\text{''}\cong F(\mathbb{Z}_d)$, the algebra of functions on the cyclic group $\mathbb{Z}_d$; or perhaps  some class of morphism $p_i\mapsto \delta_i\in F(\mathbb{Z}_d)$ be constructed, and the notion of a `normal quasi-subgroup' developed,  the contents of Section \ref{PureStates} suggest that this doesn't go anywhere useful.

\subsubsection{Pure States}\label{PureStates}

The question now arises: what is the quantum generalisation of a probability concentrated on a coset of a proper normal subgroup? An obvious generalisation of a classical $\delta_{gH}\in F(G/H)$ would be a minimal projection $p\in 2^{\mathbb{G}/\mathbb{H}}$, for a normal quantum subgroup $\mathbb{H}\lhd \mathbb{G}$ (see \cite{SZW} for definition of a \emph{normal} quantum subgroup $\mathbb{H}\lhd \mathbb{G}$; given via a quantum analogue of functions constant on cosets of $H\lhd G$).  That $\nu\in M_p(\mathbb{G})$ is concentrated on a coset  translates to $p_\nu\leq \imath(p)$, where $\imath: F(\mathbb{G}/\mathbb{H})\rightarrow F(\mathbb{G})$ is the inclusion. Minimal projections $q\in 2^\mathbb{G}$ give rise to pure states $\mathcal{F}(q/\int_\mathbb{G}q)$.

 \bigskip

The trivial subgroup $\mathbb{C}\cong F(\{e\})$ given by the surjective unital $*$-homomorphism $\varepsilon:F(\mathbb{G})\rightarrow F(\{e\})$ is a normal subgroup, and, as for all $f\in F(\mathbb{G})$:
$$(I_{F(\mathbb{G})}\otimes \varepsilon)\circ \Delta(f)=f\cong f\otimes \mathds{1}_{\{e\}},$$
\emph{all} elements of $F(\mathbb{G})$ are constant on cosets of $\{e\}\lhd \mathbb{G}$, and therefore $F(\mathbb{G})\cong F(\mathbb{G}/\{e\})$. Freely identifying these algebras, take a pure state $\delta\in M_p(\mathbb{G})$ and its associated minimal projection, which is necessarily its support $p_\delta$, and also a minimal projection in $2^{\mathbb{G}/\{e\}}$, and so `concentrated on a coset of a proper normal subgroup'.  If the random walk on $\mathbb{G}$ given by a pure state $\delta$ were ergodic, then this would be a counterexample to the claim that $\nu\in M_p(\mathbb{G})$ being supported on a coset of a proper normal subgroup is a barrier to ergodicity.

 \bigskip

 Consider the algebra of functions on $\widehat{S_3}$. Let $\rho:S_3\rightarrow \operatorname{GL}(\mathbb{C}^2)$ be the irreducible unitary two-dimensional representation, and $\xi=(1,\sqrt{2})/\sqrt{3}\in\mathbb{C}^2$ a unit vector. This data defines a state $u\in M_p(\widehat{S_3})$:
 $$u(\sigma)=\langle\xi,\rho(\sigma)\xi\rangle.$$
 Explicit calculations show that:
$$u=\delta_e+\frac{2\sqrt{2}}{3}\delta_{(12)}-\frac{\sqrt{2}}{3}\delta_{(23)}-\frac{\sqrt{2}}{3}\delta_{(13)}+\frac{e^{-5\pi/6}}{\sqrt{3}}\delta_{(123)}+\frac{e^{5\pi/6}}{\sqrt{3}}\delta_{(132)}.$$
 Note that as $u^{\star k}=u^k$, $u^{\star k}\rightarrow \delta_e=\int_{\widehat{S_3}}$, in other words the random walk given by $u$ is ergodic. However, as $\rho$ is irreducible, $u$ is a pure state. Therefore the classical condition for ergodicity, that $\nu$ not be concentrated on a coset of a proper normal subgroup, is not in general a barrier for ergodicity for random walks on quantum groups.
\section{Partial Results}
\subsection{Pure States on Kac--Paljutkin and Sekine Quantum Groups}\label{KACP}

 Recall that the algebra of functions on a finite quantum group has algebra:
$$F(\mathbb{G})\cong \bigoplus_{i=1}^N M_{n_i}(\mathbb{C}).$$
At least one of the factors must be one-dimensional to account for the counit, and to gather the one dimensional factors, reorder the index $j\mapsto i$ so that $n_i=1$ for $i=1,\dots,m_1$, and $n_i>1$ for $i>m_1$:
$$F(\mathbb{G})\cong \left(\bigoplus_{i=1}^{m_1} \mathbb{C}e_{i}\right)\oplus \bigoplus_{i=m_1+1}^N M_{n_i}(\mathbb{C})=:A_1\oplus B,$$
The pure states of $F(\mathbb{G})$ arise as pure states on single factors. If $\mathbb{G}$ is the Kac--Paljutkin quantum group $\mathfrak{G}_0$, $B=M_2(\mathbb{C})$; and if $\mathbb{G}$ is a Sekine quantum group $Y_n$, $B=M_n(\mathbb{C})$. Note further, in these cases, that
\begin{equation}\Delta(A_1)\subseteq A_1\otimes A_1+B\otimes B\text{ and }\Delta(B)\subseteq A_1\otimes B+B\otimes A_1.\qquad \text{\cites{KP6,Sekine}}\label{relation}\end{equation}
The relations above imply that if $\nu_i:=\mathcal{F}(e^i/\int_\mathbb{G} e_i)$, $p_A=\sum_{i=1}^{m_1}e_i$, and $p_B\in B$ is the identity in that matrix factor, that for all $k\in \mathbb{N}$,
$$p_{\nu_i^{\star k}}\leq p_A\Rightarrow \nu_i^{\star k}(p_B)=0,$$
and so the random walk given by $\nu_i$ is reducible and so not ergodic. The same relations imply that if $\delta \in M_{p}(\mathbb{G})$ is a pure state on $B$, $p_\delta\leq p_B$ that:
$$p_{\delta^{\star 2k}}\leq p_A\text{ and }p_{\delta^{\star (2k+1)}}\leq p_B,$$
and so the random walk given by $\delta$ is not ergodic.

\bigskip

Kac and Paljutkin \cite{KP6} show that, where $n_1$ is the number of one-dimensional factors in  $F(\mathbb{G})$, whenever $B$ consists of a single factor $M_{n_1}(\mathbb{C})$, the relations (\ref{relation}) hold, and so the random walk given by a pure state on such a quantum group is never ergodic.

\subsection{Zhang Convergence}\label{Zhang}
The following result is inspired by the classical Markov chain result that a chain with loops is aperiodic (for a random walks on a  classical group this implies $e\in\operatorname{supp }\nu$), and the proof of Zhang of this fact for the case of a Sekine quantum group (\cite{Zhang}, Proposition 4.1).
\begin{theorem}
Let $\nu\in M_p(\mathbb{G})$ be such that $\nu(p_{\varepsilon})=\nu(\eta)>0$. Then the convolution powers $(\nu^{\star k})_{k\geq 1}$ converge.
\end{theorem}
\begin{proof}
Consider the direct sum decomposition:

$$M_p(\mathbb{G})\subset F(\mathbb{G})^*=(\mathbb{C}\eta\oplus \ker \varepsilon)^*=\mathbb{C}\varepsilon \oplus (\ker \varepsilon)^*,$$
so that
$$\nu=\nu(\eta)\varepsilon+\psi,$$
with $\nu(\eta)>0$. Note that $\varepsilon$ is an idempotent state with density $f_\varepsilon=\eta/\int_\mathbb{G}\eta$.

Therefore

$$f_\nu=\frac{\nu(\eta)}{\int_\mathbb{G}\eta}\eta+f_\psi\in \mathbb{C}\eta\oplus \ker \varepsilon.$$

An element in a direct sum is positive if and only if both elements are positive. The Haar element is positive and so $f_\psi\geq 0$. Assume that $f_\psi\neq 0$ (if $f_\psi=0$, then $\psi=0\Rightarrow \nu=\varepsilon\Rightarrow \nu^{\star k}=\varepsilon$ for all $k$ and so trivial convergence). As the density of a state,
$$\int_\mathbb{G}\left(\frac{\nu(\eta)}{\int_\mathbb{G} \eta}\eta+f_\psi\right)=1\Rightarrow \int_\mathbb{G} f_\psi=1-\nu(\eta).$$

Therefore let

$$f_{\tilde{\psi}}:=\frac{f_\psi}{\int_\mathbb{G} f_\psi}=\frac{f_\psi}{1-\nu(\eta)},$$

be the density of $\tilde{\psi}\in M_p(\mathbb{G})$. Now explicitly write

$$\nu=\nu(\eta)\varepsilon+(1-\nu(\eta))\tilde{\psi}.$$

This has stochastic operator

$$T_\nu=\nu(\eta)I_{F(\mathbb{G})}+(1-\nu(\eta))T_{\tilde{\psi}}.$$

Let $\lambda$ be an eigenvalue of $T_\nu$ of eigenvector $a$. This yields

$$\nu(\eta)a+(1-\nu(\eta))T_{\tilde{\psi}}(a)=\lambda a,$$

and thus

$$T_{\tilde{\psi}}a=\frac{\lambda-\nu(\eta)}{1-\nu(\eta)}a.$$

Therefore, as $a$ is also an eigenvector for $T_{\tilde{\psi}}$, and $T_{\tilde{\psi}}$ is a stochastic operator, it follows that
\begin{align*}
  \left|\frac{\lambda-\nu(\eta)}{1-\nu(\eta)}\right| & \leq 1 \\
  \Rightarrow |\lambda-\nu(\eta)| & \leq 1-\nu(\eta).
\end{align*}
This means that the eigenvalues of $T_\nu$ lie in the ball $B_{1-\nu(\eta)}(\nu(\eta))$ and thus the only eigenvalue of magnitude one is $\lambda=1$. By the discussions of Section \ref{Spectral}, this implies that $(T_\nu^k)_{k\geq 1}$ converges and thus so does $(\nu^{\star k})_{k\geq 1}$.
\end{proof}

\subsection{Freslon's Ergodic Theorem for Random Walks on Duals\label{AmaurySec}}
In \cite{Amaury3}, Amaury Freslon proves the ergodic theorem for random walks on the duals of (possibly infinite) discrete groups. Here is the finite version:
\begin{proposition}(\cite{Amaury3}, Proposition 3.1)
  A random walk $u\in M_p(\widehat{G})$ on a finite dual group is ergodic if and only if $u$ does not coincide with a character on a non-trivial subgroup $H<G$.
\end{proposition}
The Ergodic Theorem \ref{Erg} allows us to recover Freslon's Ergodic Theorem in the finite case.

\bigskip

Let $u\in M_p(\widehat{G})$, which satisfies $u(\mathds{1}_{\widehat{G}})=1\Rightarrow u(\delta^e)\cong u(e)=1$, and also $|u(s)|\leq 1$.  Suppose that $u$ is concentrated on a proper quasi-subgroup. That $u$ is concentrated on this quasi-subgroup implies
$$u(\chi_H)=\frac{1}{|H|}\sum_{h\in H}u(\delta^h)=1,$$
and this implies that $\left.u\right|_H=1$, and so $u$ coincides on $H$ with the trivial character $H\rightarrow \{1\}$.

\bigskip

Suppose now that $u$ is not concentrated on quasi-subgroup but on a cyclic coset of a quasi-subgroup. Then there exists a quasi-subgroup $p_0=\chi_{H}$ and $d\in \mathbb{N}$ such that $T_u^d(\chi_H)=T_{u^d}(\chi_H)=\chi_H$:
\begin{align*}
  T_{u^d}(\chi_H) & =\left(u^d\otimes I_{F(\widehat{G})}\right)\Delta\left(\frac{1}{|H|}\sum_{h\in H}\delta^h\right) \\
   & =\frac{1}{|H|}\sum_{h\in H}u(h)^d\delta^h=\chi_H,
\end{align*}
so that each $u(h)$ is a $d$-th root of unity. Let $u$ be defined by a unitary representation $\rho_u:G\rightarrow \operatorname{GL}(H)$ and a unit vector $\chi\in H$. For $h\in H$, following Freslon, apply  the Cauchy--Schwarz inequality:
$$|u(h)|=|\langle \rho_u(h)\xi\,\xi\rangle|\leq \|\rho_u(h)\xi\|\|\xi\|=1,$$
see it is an equality and thus $\rho_u(h)\xi=u(h)\xi$. It follows that
\begin{align*}
u(h_1h_2)&=\langle \rho_u(h_1h_2)\xi,\xi\rangle=\langle \rho_u(h_1)\rho_u(h_2)\xi,\xi\rangle=\langle \rho_u(h_1)u(h_2)\xi,\xi\rangle
\\ &=u(h_2)\langle \rho_u(h_1)\xi,\xi\rangle=u(h_1)u(h_2),
\end{align*}
that is $\left.u\right|_H$ is a character.

\subsection{Baraquin's Ergodic Theorem\label{centralstates}}
A tool used in the quantitative analysis of random walks on classical groups is the Upper Bound Lemma of Diaconis and Shahshahani \cite{DS}. This tool was extended for use with compact classical groups by Rosenthal \cite{Ros}, finite quantum groups by the author \cite{McCarthy}, and finally for random walks given by absolutely continuous states on compact quantum groups of Kac type, by Freslon \cite{Amaury}. The upper bound follows an application of the Cauchy--Schwarz inequality to:
\begin{equation}\label{UBL}
\|\nu^{\star k}-\pi\|_2^2=\sum_{\alpha\in\operatorname{Irr}(\mathbb{G})\backslash\{\tau\}} d_\alpha \left[\left(\widehat{\nu}(\alpha)^*\right)^k\widehat{\nu}(\alpha)^k\right].
\end{equation}
The map $\|\cdot\|_2:F(\widehat{\mathbb{G}})\rightarrow \mathbb{R}$ here is related to the $\mathcal{L}^2$-norm, for $\varphi\in F(\widehat{\mathbb{G}})$ by
$$\|\varphi\|_2^2:=\|f_\varphi\|_{\mathcal{L}_2}^2=\int_\mathbb{G}|f_\varphi|^2.$$
Hence the necessity that the state $\nu\in M_p(\mathbb{G})$ be absolutely continuous (i.e. have a density $f_\nu\in\mathcal{L}^1(\mathbb{G})$, automatic in the finite case).

\bigskip

The set $\operatorname{Irr}(\mathbb{G})$ is an index set for a family of pairwise-inequivalent irreducible unitary representations of the compact quantum group $\mathbb{G}$ (the representations are given by corepresentations $\kappa_\alpha:V_\alpha\rightarrow V_\alpha\otimes F(\mathbb{G})$). The index $\tau$ is for the trivial representation. The dimension $d_\alpha\in\mathbb{N}$ is the dimension of the vector space $V_\alpha$, while the linear map $\widehat{\nu}(\alpha)\in L(\overline{V})$, the Fourier transform of $\nu$ at the representation $\kappa_\alpha$, is given by:
$$\widehat{\nu}(\alpha)=(I_{\overline{V_\alpha}}\otimes\nu)\circ \overline{\kappa_\alpha}.$$
Here $\overline{\kappa_\alpha}$ is the representation conjugate to $\kappa_{\alpha}$.

\bigskip

 However, for finite quantum groups, of course, all norms are equivalent. Thus (\ref{UBL}) can be used qualitatively, to detect if the random walk given by $\nu$ is ergodic, and there are a class of states whose ergodicity can be determined quite easily via the upper bound lemma.

\bigskip

Following Freslon \cite{Amaury}, consider the \emph{central algebra} of a finite quantum group, $F(\mathbb{G})_0$, the span of the irreducible characters of $\mathbb{G}$. Where $\{\rho^{\alpha}_{ij}:i,j=1,\dots,d_\alpha\}$ are the matrix coefficients of an irreducible representation $\kappa_\alpha$, the character of $\kappa_\alpha$ is given by:
$$\chi_\alpha:=\sum_{i=1}^{d_\alpha} \rho_{ii}^\alpha\in F(\mathbb{G}),$$
so that $F(\mathbb{G})_0=\operatorname{span}\{\chi_\alpha:\alpha\in\operatorname{Irr}(\mathbb{G})\}$. Consider a state $\nu\in M_p(\mathbb{G})$ whose density $f_\nu$ is in the central algebra:
$$f_\nu=\sum_{\alpha\in\operatorname{Irr}(\mathbb{G})}f_\alpha \chi_\alpha.$$
For such states, it can be shown that the Fourier transform at a representation indexed by $\alpha$ is scalar:
$$\widehat{\nu}(\alpha)=\frac{f_\alpha}{d_\alpha}\cdot I_{d_\alpha}\Rightarrow \left(\widehat{\nu}(\alpha)^*\right)^k\widehat{\nu}(\alpha)^k=\frac{|f_\alpha|^{2k}}{d_\alpha^{2k}}\cdot I_{d_\alpha},$$
so that, for such a central state:
$$\|\nu^{\star k}-\pi\|_2^2=\sum_{\alpha\in\operatorname{Irr}(\mathbb{G})\backslash \{\tau\}}d_\alpha^2\left|\frac{f_\alpha}{d_\alpha}\right|^{2k}.$$
 When stating it for the case of a Sekine quantum group, Baraquin (\cite{Baraquin}, Proposition 3) all but wrote down the following corollary:
\begin{corollary}(Baraquin's Ergodic Theorem)\label{BET}
If a random walk on a finite quantum group $\mathbb{G}$ given by $\nu\in M_p(\mathbb{G})$ has density $f_\nu=\sum_{\alpha\in \operatorname{Irr}(\mathbb{G})}f_\alpha \chi_\alpha\in F(\mathbb{G})_0$, then it is ergodic if and only if
$$|f_\alpha|<d_\alpha,$$
for all non-trivial irreducible representations $\kappa_\alpha$.
\end{corollary}

\subsection*{Acknowledgements.} I would like to thank Uwe Franz; much progress on this problem was achieved during a (very enjoyable) May 2019 visit to Uwe at the Laboratoire de math\'{e}matiques de Besan\c{c}on ($\text{Lm}^{\text{B}}$), France. This trip was financially supported by $\text{Lm}^{\text{B}}$, and also Cork Institute of Technology.

\end{document}